\documentclass[11pt]{amsart}

\usepackage{amsmath,amssymb,amsfonts}
\usepackage{cite}
\usepackage[margin=1in]{geometry}

\newtheorem{theorem}{Theorem}[section]

\newtheorem{conjecture}[theorem]{Conjecture}
\newtheorem{corollary}[theorem]{Corollary}
\newtheorem{lemma}[theorem]{Lemma}

\theoremstyle{definition}

\newtheorem{question}[theorem]{Question}

\theoremstyle{remark}
\newtheorem{remark}[theorem]{Remark}

\newcommand{\cT}{\mathcal T}
\newcommand{\cF}{\mathcal F}
\newcommand{\eps}{\varepsilon}

\allowdisplaybreaks[4]

\title[Resilient forest universality]{Resilient forest universality in percolated dense graphs}

\author{Mostafa Mirabi}
\address{The Taft School, Watertown, Connecticut, USA; and Wesleyan University, Middletown, Connecticut, USA}
 \email{mmirabi@wesleyan.edu}
\urladdr{https://sites.google.com/site/mostafamirabi/}

\begin{document}

\begin{abstract}
Christoph, M\"uyesser and Wigderson recently asked whether an approximate form of the Erd\H{o}s--S\'os conjecture is robust under random edge deletions. We establish a density-sensitive transference theorem that converts global resilience for bounded-degree trees in sparse random graphs into resilient forest universality in arbitrary dense host graphs. More precisely, if $F$ is an $N$-vertex graph of edge density $\lambda$ bounded away from zero and $p\in[K/N,1]$, then, with probability $1-o(1)$ uniformly over the host and the percolation parameter, every subgraph obtained from $F_p$ by deleting at most an $\alpha$-fraction of its edges contains every bounded-degree forest on at most $((1-\alpha)\lambda-\xi)N$ vertices. Consequently, for every $c>0$, $L\ge1$, fixed $D$ and $\alpha<c$, every graph $F$ on at most $Ld$ vertices with average degree at least $d$ has the property that, after percolation at any rate $p\in[K/d,1]$ and any subsequent deletion of at most an $\alpha$-fraction of the surviving edges, the remaining graph is universal for all forests on at most $(1-c)d$ vertices and maximum degree at most $D$. This includes vertex-disjoint packings of any prescribed collection of bounded-degree trees of that total order. We further obtain forest-universality results for graphs whose connected components have vertex-cover number at most $Cd$, and for graphs that can be brought into this form by deleting sufficiently few edges on the $dn$-scale.
\end{abstract}

\maketitle

\section{Introduction}

A basic theme in extremal graph theory is that average degree forces large structured subgraphs. The Erd\H{o}s--Gallai theorem says that every graph of average degree $d$ contains a cycle of length at least $d$ \cite{erdos1959maximal}. The tree analogue is the Erd\H{o}s--S\'os conjecture, which predicts that average degree greater than $k-1$ forces a copy of every tree with $k$ edges \cite{erdos_sos1970some}. These two statements have the same extremal scale, but the tree problem is much more sensitive to the structure of the host graph.

A natural robustness question asks whether such deterministic conclusions survive random edge deletions. Throughout, all graphs are finite and simple. Given a graph $G$ and $p\in[0,1]$, write $G_p$ for the random subgraph obtained by keeping each edge independently with probability $p$. Christoph, M\"uyesser and Wigderson proved that the Erd\H{o}s--Gallai theorem is robust in this sense: for every $c>0$, if $G$ has average degree $d$ and $p\ge K(c)/d$, then $G_p$ asymptotically almost surely contains a cycle of length at least $(1-c)d$ \cite{christoph2026robustness}. They then proposed the following tree analogue.

\begin{conjecture}\label{conj:cmw}
For every $c>0$ there is $K=K(c)$ such that the following holds. Let $G$ be an $n$-vertex graph with average degree $d$, and let $T$ be a tree on $\lfloor(1-c)d\rfloor$ vertices with maximum degree at most $\Delta$. If $p\in[0,1]$ and
$
        p\ge K\Delta/d,
$
then $G_p$ contains a copy of $T$ with probability tending to $1$ as $n\to\infty$.
\end{conjecture}

This question sits between several well-studied problems. When $p=1$, it becomes an approximate Erd\H{o}s--S\'os statement. In dense hosts, the bounded-degree case was proved by Besomi, Pavez-Sign\'e and Stein \cite{besomi2021erdos}, while Rozho\v{n} developed a local approach that gives approximate results under degree hypotheses \cite{rozhon2019local}. Very recently, Davoodi, Piguet, \v{R}ada and Sanhueza-Matamala proved a dense asymptotic version in substantially greater generality \cite{davoodi2026asymptotic}. On the random side, universality for bounded-degree trees has been developed through expansion and random-graph methods, starting from the Friedman--Pippenger tree-embedding criterion \cite{friedman1987expanding} and including the work of Alon, Krivelevich and Sudakov \cite{alon2007embedding}, Balogh, Csaba, Pei and Samotij \cite{balogh2010large}, and Montgomery \cite{montgomery2019spanning}. The novelty in Conjecture~\ref{conj:cmw} is the combination of random robustness with only an average-degree assumption on the host.

The purpose of this paper is to establish Conjecture~\ref{conj:cmw} for fixed maximum degree in the first natural structural regime: dense hosts of order $O(d)$, and consequently graphs whose components have vertex covers of order $O(d)$. Our starting point is a density-sensitive transference principle. It shows that the proportion of a dense host retained after an adversarial deletion directly determines the order of the bounded-degree forests that remain universal after random sparsification. The resulting statements are stronger than tree containment: they give universality for forests and hence vertex-disjoint packings of arbitrary prescribed collections of bounded-degree trees.

The dense and small-cover regimes are important for two reasons. First, dense $O(d)$-vertex graphs are the finite-density model at the Erd\H{o}s--S\'os scale; cliques and complete bipartite graphs with one side of order $d$ are the basic examples. Second, small-cover components are one of the main structural outcomes in the hyperstability theorem of Christoph, M\"uyesser and Wigderson \cite{christoph2026robustness}, and related small-cover structures also appear in Pokrovskiy's hyperstability approach to bounded-degree trees \cite{pokrovskiy2024hyperstability}. Thus the small-cover case is not an artificial dense special case; it is the stable part of the general average-degree problem.

For a real number $m\ge0$ and an integer $D\ge0$, let $\cT_{\le m,D}$ and $\cF_{\le m,D}$ denote, respectively, the families of all trees and all forests with at most $\lfloor m\rfloor$ vertices and maximum degree at most $D$. Thus expressions such as $\cF_{\le(1-c)d,D}$ always use downward rounding. A graph is $\mathcal A$-universal if it contains every member of $\mathcal A$ as a not necessarily induced subgraph.

All host graphs in the probabilistic statements are deterministic, and probability is taken only over the independent edge choices defining the relevant percolated graph. In particular, when a statement says that with high probability every subgraph $H$ has a given property, the quantifier over $H$ is inside the high-probability event and $H$ may be chosen after the percolated graph is exposed. For assertions involving a sequence $G_n$, asymptotically almost surely means that the probability tends to one for every deterministic sequence of graphs and parameters satisfying the displayed hypotheses.

Our density-sensitive theorem is stated and proved in Section~\ref{sec:transference}. Its principal consequence at the Erd\H{o}s--S\'os scale is the following.

\begin{theorem}\label{thm:dense}
For every $c\in(0,1)$, $L\ge1$, integer $D\ge1$ and $\alpha\in[0,c)$, there are constants $K=K(c,L,D,\alpha)$ and $d_0=d_0(c,L,D,\alpha)$, and a function $\eps:[d_0,\infty)\to[0,1]$ with $\eps(d)\to0$ as $d\to\infty$, such that the following holds. For every real $d\ge d_0$, every $N$-vertex graph $F$ with $N\le Ld$ and average degree at least $d$, and every $p\in[K/d,1]$,
\[
\mathbb P\left(
\begin{array}{c}
\text{every subgraph $H\subseteq F_p$ with }e(H)\ge(1-\alpha)e(F_p)\\[2pt]
\text{is $\cF_{\le(1-c)d,D}$-universal}
\end{array}
\right)\ge1-\eps(d).
\]
The same function $\eps$ works simultaneously for all choices of $N$, $F$ and $p$ in the indicated ranges.
\end{theorem}

Taking $\alpha=0$ gives the robust dense-host statement proposed above. Thus Theorem~\ref{thm:dense} gives global edge resilience in addition to random robustness, while forest universality gives the corresponding tree-packing statement.

\begin{corollary}\label{cor:packings}
Under the hypotheses and on the high-probability event of Theorem~\ref{thm:dense}, every eligible subgraph $H$ contains vertex-disjoint copies of any prescribed collection $T_1,\ldots,T_s$ of trees satisfying
\[
        \sum_{i=1}^s |V(T_i)|\le(1-c)d
        \qquad\text{and}\qquad
        \max_{1\le i\le s}\Delta(T_i)\le D.
\]
\end{corollary}

\begin{proof}
The disjoint union of $T_1,\ldots,T_s$ belongs to $\cF_{\le(1-c)d,D}$, so the conclusion follows directly from Theorem~\ref{thm:dense}.
\end{proof}

The dense theorem implies the small-cover case of Conjecture~\ref{conj:cmw} for bounded-degree trees.

\begin{theorem}\label{thm:smallcover}
For every $c\in(0,1)$, $C\ge1$ and integer $D\ge1$, there are constants $K=K(c,C,D)$ and $d_0=d_0(c,C,D)$ such that the following holds. For every sequence $(G_n,d_n,p_n)_{n\ge1}$ in which $G_n$ is an $n$-vertex graph of average degree at least $d_n\ge d_0$, every connected component of $G_n$ has vertex-cover number at most $Cd_n$, and $p_n\in[K/d_n,1]$, one has
\[
\mathbb P\bigl((G_n)_{p_n}\text{ is $\cF_{\le(1-c)d_n,D}$-universal}\bigr)\longrightarrow1
\qquad\text{as }n\to\infty.
\]
\end{theorem}

\subsection*{Relation to earlier work and sharpness}

Theorem~\ref{thm:transference} uses the global resilience theorem of Ara\'ujo, Moreira and Pavez-Sign\'e \cite[Theorem~1.3]{araujo2023ramsey} as a black box. Their theorem concerns the ambient binomial random graph $\Gamma=G(N,p)$: a subgraph retaining more than a $(\rho+\delta)$-fraction of $e(\Gamma)$ is universal for all bounded-degree trees with $\lfloor\rho N\rfloor$ edges. In the present setting the random graph is $F_p=F\cap\Gamma$, where $F$ is an arbitrary deterministic host of density $\lambda$. The transference step is the uniform conversion
\[
        \frac{e(F\cap\Gamma)}{e(\Gamma)}=\lambda+o(1),
\]
which turns retention of a $(1-\alpha)$-fraction of $F_p$ into relative density $(1-\alpha)\lambda+o(1)$ inside $\Gamma$. Together with the elementary forest-extension lemma, this gives a host-relative statement not contained explicitly in \cite{araujo2023ramsey}: the density-sensitive order $((1-\alpha)\lambda-\xi)N$, universality for all smaller forests, and the prescribed tree-packing conclusion.

Christoph, M\"uyesser and Wigderson prove random robustness for long cycles in arbitrary average-degree hosts \cite[Theorem~1.1]{christoph2026robustness} and formulate Conjecture~\ref{conj:cmw} as their Conjecture~9.1. Theorem~\ref{thm:dense} verifies the fixed-maximum-degree version only in the dense regime $N=O(d)$, while Theorem~\ref{thm:smallcover} and Corollary~\ref{cor:nearsmallcover} treat the small-cover and near-small-cover structures highlighted by their hyperstability theorem \cite[Theorem~1.2]{christoph2026robustness}. Thus the present results do not settle the arbitrary-host conjecture. Their extra conclusions in the regimes considered here are global edge resilience after percolation and simultaneous universality for every bounded-degree forest of the prescribed order.

In the deterministic specialization $p=1$ and $\alpha=0$, the dense forest-universality conclusion follows, after applying Lemma~\ref{lem:forestextension} and adjusting the fixed slack, from the dense asymptotic Erd\H{o}s--S\'os theorem of Davoodi, Piguet, \v{R}ada and Sanhueza-Matamala \cite[Corollary~1.4]{davoodi2026asymptotic}; their result has no bounded-maximum-degree restriction. The contribution here is therefore complementary rather than a new deterministic dense embedding theorem: it concerns sparse percolation at the scale $p=\Theta(1/d)$, adversarial deletion after percolation, and a uniform host-relative formulation. Pokrovskiy \cite[Theorem~1.3]{pokrovskiy2024hyperstability} proves a deterministic hyperstability theorem which, under its density hypothesis, turns the absence of a fixed bounded-degree tree into a decomposition with vertex covers of order $O(d)$ after deleting few edges. We assume such small-cover structure rather than derive it; Theorem~\ref{thm:smallcover} and Corollary~\ref{cor:nearsmallcover} show that these stable pieces retain bounded-degree forest universality after random sparsification.

Several scales in the main statements are best possible up to constants and the fixed linear slack. First, the order $p=1/d$ cannot be lowered: for $F=K_{\lceil d\rceil+1}$ and $p=o(1/d)$, one has $\mathbb E e(F_p)=o(d)$, so Markov's inequality shows that the probability that $F_p$ contains a tree on $\gamma d$ vertices tends to zero for every fixed $\gamma>0$. This is the same threshold obstruction noted for the ambient random-graph resilience theorem in \cite{araujo2023ramsey}. Second, the coefficient of the host density in Theorem~\ref{thm:transference} cannot be uniformly increased. If $F$ is the disjoint union of $r$ almost equal cliques, then $\lambda=1/r+o(1)$ and every component has order at most $(\lambda+o(1))N$. Third, for $F=K_N$, retaining only the percolated edges inside an almost balanced $r$-partition keeps a $(1/r+o(1))$-fraction of the edges and leaves every component of order at most $(1/r+o(1))N$; this is also the sharpness construction from \cite{araujo2023ramsey}. Hence the factor $1-\alpha$ is asymptotically sharp at the reciprocal values $1-\alpha=1/r$, and the relation $\alpha<c$ in Theorem~\ref{thm:dense} cannot be reversed uniformly. Finally, disjoint unions of cliques of order $d+1$ show that the forest-order scale $d$ itself is asymptotically optimal. We do not claim optimal dependence of $K$ on the fixed parameters, and the removal of the slack $\xi$ from Theorem~\ref{thm:transference} is not addressed. The fixed-degree hypothesis is also a limitation: Conjecture~\ref{conj:cmw} predicts the more general threshold $p=\Theta(\Delta/d)$, while \cite{davoodi2026asymptotic} removes the degree restriction when $p=1$.

We close the introduction with a brief proof overview. We view $F_p$ as the intersection of $F$ with a binomial random graph $\Gamma\sim G(N,p)$. The global resilience theorem of Ara\'ujo, Moreira and Pavez-Sign\'e \cite{araujo2023ramsey}, combined with uniform concentration of the relative density of $F\cap\Gamma$, gives the transference theorem. An elementary extension lemma then upgrades tree universality to forest universality. Finally, the small-cover theorem follows from a cover-to-dense reduction: a graph of average degree $d$ with a vertex cover of size $O(d)$ contains a subgraph on $O(d)$ vertices whose average degree is still close to $d$.
\vspace{.2cm}

Acknowledgments: 
The author would like to thank the anonymous referee for their very helpful comments that improved the paper.\\
AI Use Statement: ChatGPT 5.5 was used to proofread the paper, improve its language, and assist with the calculations in Theorem 3.1.

%%%%%%%%%%%%%%%

\section{Preliminaries}
We use one known result and two elementary observations. The known result is a global resilience theorem for bounded-degree trees in sparse random graphs due to Ara\'ujo, Moreira and Pavez-Sign\'e \cite{araujo2023ramsey}. They state it for trees with a prescribed number of edges, and we keep the resulting floor in the statement below.

\begin{theorem}\label{thm:amps}
For every integer $D\ge3$ and every $\rho,\delta\in(0,1)$, there is $C>0$ such that if $p\in[C/N,1]$, then $\Gamma=G(N,p)$ has the following property with high probability: every subgraph $\Gamma'\subseteq \Gamma$ satisfying
$
        e(\Gamma')> (\rho+\delta)e(\Gamma)
$
is universal for the family of all trees with $\lfloor \rho N\rfloor$ edges and maximum degree at most $D$.
\end{theorem}

We first record the elementary observation that tree universality automatically yields forest universality after a harmless change in the degree bound.

\begin{lemma}\label{lem:forestextension}
Every forest $J$ on at least one vertex can be extended, by adding edges between its components, to a tree $J^+$ on the same vertex set such that
$
        \Delta(J^+)\le \max\{\Delta(J),2\}.
$
\end{lemma}

\begin{proof}
Let $J_1,\ldots,J_s$ be the components of $J$. For every nontrivial component $J_i$, choose two distinct leaves $a_i$ and $b_i$; if $J_i$ is an isolated vertex, set $a_i=b_i$ equal to that vertex. Add the edges
$
        b_ia_{i+1},$\, $ 1\le i<s.
$
The resulting graph $J^+$ is connected and acyclic, and hence is a tree. Each chosen leaf acquires at most one new incident edge, while an isolated vertex acquires at most two. Therefore $\Delta(J^+)\le\max\{\Delta(J),2\}$.
\end{proof}

The final input is a reduction from small vertex cover to dense bounded-order subgraphs. It is similar to the cover-to-dense lemma used in \cite{christoph2026robustness}; we include the short proof for completeness.

\begin{lemma}\label{lem:coverdense}
For every $C\ge1$ and $\eta>0$, there are constants $L=L(C,\eta)$ and $a_0=a_0(C,\eta)$ such that the following holds for all $a\ge a_0$. If $H$ is a graph with average degree at least $a$ and vertex-cover number at most $Ca$, then $H$ contains a subgraph $F$ such that
$
        |V(F)|\le La
        $ and $
        \overline d(F)\ge (1-\eta)a.
$
\end{lemma}

\begin{proof}
If $\eta\ge1$, take $L=2$ and $a_0=1$. Since $H$ has positive average degree, it contains an edge, and that edge alone gives a subgraph of order at most $La$ and average degree at least $(1-\eta)a$. We may therefore assume that $\eta\in(0,1)$.

Let $X$ be a vertex cover of $H$ with $|X|=x\le Ca$, and put $Y=V(H)\setminus X$. Then $Y$ is independent. Choose a constant
$
        M>\frac{C(1-\eta)}{\eta},
$
set $L=C+M+1$, and choose $a_0\ge1$ so that
\[
        \frac{M}{C+M+1/a_0}\ge1-\eta.
\]
If $|Y|\le Ma+1$, then $F:=H$ has at most $(C+M+1)a=La$ vertices and average degree at least $a$, so we are done.

Suppose now that $|Y|>Ma+1$. Let $Y'\subseteq Y$ be a set of $\lceil Ma\rceil$ vertices with largest degree into $X$, and let $F:=H[X\cup Y']$. Since the sum of the $|Y'|$ largest degrees from $Y$ into $X$ is at least a $|Y'|/|Y|$ proportion of the total, we have
\[
        e(F)\ge e_H(X)+\frac{|Y'|}{|Y|}e_H(X,Y)
        \ge \frac{|Y'|}{|Y|}e(H).
\]
Using $e(H)\ge a(x+|Y|)/2$, it follows that
\[
        e(F)\ge \frac{|Y'|}{|Y|}\cdot \frac{a(x+|Y|)}2
        \ge \frac{a|Y'|}{2}.
\]
Therefore
\[
        \overline d(F)\ge \frac{a|Y'|}{x+|Y'|}
        \ge \frac{aMa}{Ca+Ma+1}.
\]
By the choice of $a_0$, the last expression is at least $(1-\eta)a$ for every $a\ge a_0$. Also $|V(F)|=x+|Y'|\le La$. This proves the lemma.
\end{proof}

\section{Density-sensitive transference}\label{sec:transference}

We now prove the general transference theorem. It records the full dependence of the guaranteed forest order on the density of the host and on the proportion of edges retained after the adversarial deletion.

\begin{theorem}\label{thm:transference}
For every $\beta\in(0,1]$, $\alpha\in[0,1)$, integer $D\ge1$ and $\xi\in(0,(1-\alpha)\beta)$, there are a constant $K=K(\beta,\alpha,D,\xi)$, an integer $N_0=N_0(\beta,\alpha,D,\xi)$ and a function $\eps:\mathbb N\to[0,1]$ with $\eps(N)\to0$ as $N\to\infty$ such that the following holds. For every $N\ge N_0$, every $N$-vertex graph $F$ of edge density
\[
        \lambda:=\frac{e(F)}{\binom N2}\ge\beta,
\]
and every $p\in[K/N,1]$,
\[
\mathbb P\left(
\begin{array}{c}
\text{every subgraph $H\subseteq F_p$ with }e(H)\ge(1-\alpha)e(F_p)\\[2pt]
\text{is $\cF_{\le((1-\alpha)\lambda-\xi)N,D}$-universal}
\end{array}
\right)\ge1-\eps(N).
\]
The function $\eps$ is independent of the choices of $F$ and $p$ in these ranges.
\end{theorem}

\begin{proof}
Put $D_0:=\max\{D,3\}$ and let
$
        I=[(1-\alpha)\beta,1-\alpha].
$
Choose a finite set $\mathcal R\subset(0,1)$ such that for every $x\in I$ there is some $\rho\in\mathcal R$ satisfying
\begin{equation}\label{eq:rhochoice}
        x-\frac{3\xi}{4}<\rho<x-\frac{\xi}{2}.
\end{equation}
Such a finite set exists because the intervals in \eqref{eq:rhochoice} have fixed positive length and their centers range over the compact interval $I$. For each $\rho\in\mathcal R$, apply Theorem~\ref{thm:amps} with parameters $D_0,\rho,\xi/8$, and let $C_\rho$ be the resulting constant. Choose
$
        K>\max_{\rho\in\mathcal R}C_\rho.
$

For $\rho\in\mathcal R$, let $\mathcal E_\rho(N,p)$ be the event that the conclusion of Theorem~\ref{thm:amps} holds for $\Gamma=G(N,p)$ with parameters $D_0,\rho,\xi/8$. The high-probability assertion in Theorem~\ref{thm:amps} is uniform over $p\in[K/N,1]$: more precisely,
\[
        u_\rho(N):=\sup_{p\in[K/N,1]}\mathbb P\bigl(\mathcal E_\rho(N,p)^c\bigr)
        \longrightarrow0.
\]
Indeed, otherwise there would be a constant $a>0$, a subsequence $N_j\to\infty$ and choices $p_j\ge K/N_j$ for which the failure probability is at least $a$. Extending the $p_j$ to a sequence $p=p(N)\ge K/N$ would contradict Theorem~\ref{thm:amps}. Since $\mathcal R$ is finite,
\[
        u(N):=\sum_{\rho\in\mathcal R}u_\rho(N)=o(1).
\]

Now expose $\Gamma\sim G(N,p)$ on the vertex set of $F$, so that $F_p$ has the same distribution as $F\cap\Gamma$. Set $M:=\binom N2$ and $t_N:=N^{-1/4}$. The random variables $e(\Gamma)$ and $e(F\cap\Gamma)$ are binomial with means
$
        \mu_0=pM
        $ and $
        \mu_F=pe(F)=~p\lambda M,
$
respectively. Uniformly over all $p\in[K/N,1]$ and all $F$ with $\lambda\ge\beta$,
\[
        \mu_0\ge\frac{K(N-1)}2
        \qquad\text{and}\qquad
        \mu_F\ge\frac{\beta K(N-1)}2.
\]
Therefore the multiplicative Chernoff bound gives
\begin{align*}
&\mathbb P\left(
 |e(\Gamma)-\mu_0|>t_N\mu_0
 \text{ or }
 |e(F\cap\Gamma)-\mu_F|>t_N\mu_F
 \right)\\
&\hspace{35mm}\le
4\exp\left(-\frac{\beta Kt_N^2(N-1)}6\right)
=:v(N)=o(1),
\end{align*}
where the bound is independent of $F$ and $p$. On the complementary event, write
\[
        e(\Gamma)=\mu_0(1+\theta_0),
        \qquad
        e(F\cap\Gamma)=\mu_F(1+\theta_F),
        \qquad
        |\theta_0|,|\theta_F|\le t_N.
\]
Then
\begin{equation}\label{eq:relative-density}
\left|
        \frac{e(F\cap\Gamma)}{e(\Gamma)}-\lambda
\right|
=\lambda\left|\frac{1+\theta_F}{1+\theta_0}-1\right|
\le\frac{2t_N}{1-t_N}
=:r_N,
\end{equation}
where $r_N=o(1)$ uniformly over the required range.

Set $x:=(1-\alpha)\lambda\in I$, and choose $\rho\in\mathcal R$ satisfying \eqref{eq:rhochoice}. On the intersection of the events above, every $H\subseteq F\cap\Gamma$ with $e(H)\ge(1-\alpha)e(F\cap\Gamma)$ satisfies, for all sufficiently large $N$,
\[
        \frac{e(H)}{e(\Gamma)}
        \ge (1-\alpha)\frac{e(F\cap\Gamma)}{e(\Gamma)}
        \ge x-(1-\alpha)r_N
        >\rho+\frac{\xi}{8}.
\]
Theorem~\ref{thm:amps} therefore implies that $H$ contains every tree with $\lfloor\rho N\rfloor$ edges and maximum degree at most $D_0$.

Let $J\in\cF_{\le(x-\xi)N,D}$. The empty forest is trivial, so assume $J$ has at least one vertex. By Lemma~\ref{lem:forestextension}, $J$ extends to a tree $J^+$ on the same vertex set with maximum degree at most $D_0$. Since \eqref{eq:rhochoice} gives $\rho>(x-\xi)+\xi/4$, for all sufficiently large $N$ we have
\[
        e(J^+)\le |V(J)|-1
        \le\lfloor(x-\xi)N\rfloor-1
        \le\lfloor\rho N\rfloor.
\]
Extend $J^+$ further, if necessary, to a tree with exactly $\lfloor\rho N\rfloor$ edges by attaching a path at a vertex of degree at most one. This is possible because $\rho<1$, so a tree with $\lfloor\rho N\rfloor$ edges has at most $N$ vertices. The maximum degree remains at most $D_0$. This larger tree embeds in $H$, and hence so does $J$.

Choose $N_0\ge K$ large enough for all of the preceding estimates and define
\[
        \eps(N):=\min\{1,u(N)+v(N)\}
\]
for $N\ge N_0$ (and arbitrarily for smaller $N$). Then $\eps(N)\to0$, the bound is independent of $F$ and $p$, and the event just used is independent of the choices of $H$ and $J$. This proves simultaneous universality for every eligible $H$.
\end{proof}

The dense-host result is now an immediate consequence.

\begin{proof}[Proof of Theorem~\ref{thm:dense}]
Let
$
        \beta:=\frac1L
        $ and $
        \xi:=\frac{c-\alpha}{2L}.
$
Since $c<1$, we have $0<\xi<(1-\alpha)\beta$, as required in Theorem~\ref{thm:transference}. Let $K$, $N_0$ and $\eps_0(N)$ be supplied by that theorem for these parameters, and choose $d_0\ge\max\{K,N_0\}$.

The hypotheses imply $N\ge d+1$ and, on writing $q=d/N$, that $q\ge1/L$. The edge density $\lambda$ of $F$ satisfies
\[
        \lambda\ge\frac{d}{N-1}>q\ge\beta.
\]
Moreover,
\[
        ((1-\alpha)\lambda-\xi)N
        >(1-\alpha)d-\frac{c-\alpha}{2L}N
        \ge\left(1-\frac{c+\alpha}{2}\right)d
        >(1-c)d.
\]
Since $p\ge K/d$ and $N>d$, we have $p\ge K/N$, so Theorem~\ref{thm:transference} applies. Define the tail supremum
\[
        \eps(d):=\sup\bigl(\{0\}\cup\{\eps_0(N):N\in\mathbb N,\ N>d\}\bigr).
\]
Because $\eps_0(N)\to0$, its tail supremum also tends to zero. For every eligible $F$ we have $N>d$, so its failure probability is at most $\eps_0(N)\le\eps(d)$. This proves the claimed uniformity over all $N$, $F$ and $p$ in Theorem~\ref{thm:dense}.
\end{proof}

%%%%%%%%%%%%%%%%%%%
%%%%%%%%%%%%%%%%%%%%%%%%%%%

\section{Small-cover components}

We now prove Theorem~\ref{thm:smallcover}. The proof extracts edge-disjoint dense spots. For large $d$, each spot succeeds with probability tending to one; when $d$ remains bounded, there are linearly many independent spots, each with probability bounded away from zero.

\begin{proof}[Proof of Theorem~\ref{thm:smallcover}]
Fix an arbitrary sequence satisfying the hypotheses of the theorem, and abbreviate $G:=G_n$, $d:=d_n$ and $p:=p_n$. Choose $\eta>0$ sufficiently small that
\begin{equation}\label{eq:etafit}
        1-c \le (1-c/2)(1-\eta)^2.
\end{equation}
Set
\[
        C':=\frac{C}{1-\eta}.
\]
Apply Lemma~\ref{lem:coverdense} with parameters $C'$ and $\eta$, and let $L_0=L(C',\eta)$ and $a_0=a_0(C',\eta)$ be the resulting constants. Define
\[
        d_*:=(1-\eta)^2d
        \qquad\text{and}\qquad
        L:=\frac{L_0}{(1-\eta)^2}.
\]
Apply Theorem~\ref{thm:dense} with parameters $c/2,L,D$ and $\alpha=0$, and let $K_0$ and $d_1$ be the constants it gives. Increase $d_1$ if necessary so that Lemma~\ref{lem:coverdense} is applicable at scale $(1-\eta)d$ whenever $d\ge d_1$. Finally choose
\[
        K:=\frac{K_0}{(1-\eta)^2}
\]
and take $d_0$ large enough for all previous requirements. If $p\ge K/d$, then $p\ge K_0/d_*$.

We construct a family of edge-disjoint subgraphs. Start with the residual graph $R_0:=G$. As long as the current residual graph $R_j$ has at least $(1-\eta)dn/2$ edges, some connected component $H_j$ of $R_j$ has average degree at least $(1-\eta)d$. The component $H_j$ is contained in a component of $G$, and hence has vertex-cover number at most $Cd=C'(1-\eta)d$. By Lemma~\ref{lem:coverdense}, applied with $a=(1-\eta)d$, the graph $H_j$ contains a subgraph $F_j$ such that
\begin{equation}\label{eq:Fj-size}
        |V(F_j)|\le L_0(1-\eta)d\le L_0d=Ld_*,
\end{equation}
and
\begin{equation}\label{eq:Fj-degree}
        \overline d(F_j)\ge (1-\eta)^2d=d_*.
\end{equation}
Remove all edges of $F_j$ from the residual graph and continue.

When the process stops, fewer than $(1-\eta)dn/2$ edges remain. Since $G$ initially has at least $dn/2$ edges, the extracted graphs contain at least $\eta dn/2$ edges in total. Each $F_j$ has at most $L_0d$ vertices and therefore at most $(L_0d)^2/2$ edges. Hence, if $r$ is the number of extracted subgraphs, then
\begin{equation}\label{eq:r-lower}
        r\ge \frac{\eta n}{L_0^2d}.
\end{equation}

For every $j$, \eqref{eq:Fj-size} and \eqref{eq:Fj-degree} allow us to apply Theorem~\ref{thm:dense} to $F_j$ at scale $d_*$. Moreover, by \eqref{eq:etafit},
$
        (1-c)d\le (1-c/2)d_*.
$
Thus Theorem~\ref{thm:dense} implies that $F_{j,p}$ is $\cF_{\le(1-c)d,D}$-universal with probability at least $1-\eps(d_*)$, where $\eps(x)\to0$ is the error term from Theorem~\ref{thm:dense}. In particular, after increasing $d_0$ if necessary, this probability is at least $1/2$ for all $d\ge d_0$.

The graphs $F_j$ are edge-disjoint, so the corresponding percolation events are independent. Let $\pi(d):=1-\eps(d_*)$. Then $\pi(d)\to1$ as $d\to\infty$ and $\pi(d)\ge1/2$ for all $d\ge d_0$. Therefore the probability that none of the $F_j$ is universal is at most
$
        (1-\pi(d))^r.
$
To verify that this bound tends to zero uniformly for every choice of $d=d(n)$, fix $\gamma>0$. Choose $d_2\ge d_0$ so large that $1-\pi(d)\le\gamma$ whenever $d\ge d_2$. If $d\ge d_2$, then $r\ge1$ and the failure probability is at most $\gamma$. If $d<d_2$, then \eqref{eq:r-lower} and $\pi(d)\ge1/2$ give
\[
        (1-\pi(d))^r
        \le 2^{-r}
        \le 2^{-\eta n/(L_0^2d_2)}=o(1).
\]
Since $\gamma$ is arbitrary, the failure probability tends to zero. Hence, with probability tending to one, at least one extracted subgraph is $\cF_{\le(1-c)d,D}$-universal. Since each $F_j$ is a subgraph of $G$, the theorem follows.
\end{proof}

The small-cover conclusion is stable under deleting a small number of edges before percolation. This gives a form that is closer to the structural conclusions of hyperstability results.

\begin{corollary}\label{cor:nearsmallcover}
For every $c\in(0,1)$, $C\ge1$, integer $D\ge1$ and $\delta\in[0,c/2)$, there are constants $K=K(c,C,D,\delta)$ and $d_0=d_0(c,C,D,\delta)$ such that the following holds. For every sequence $(G_n,G_n^0,d_n,p_n)_{n\ge1}$ in which $G_n$ is an $n$-vertex graph of average degree at least $d_n\ge d_0$, $G_n^0$ is a spanning subgraph satisfying
$
        e(G_n)-e(G_n^0)\le \delta d_n n,
$
every connected component of $G_n^0$ has vertex-cover number at most $Cd_n$, and $p_n\in[K/d_n,1]$, one has
\[
\mathbb P\bigl((G_n)_{p_n}\text{ is $\cF_{\le(1-c)d_n,D}$-universal}\bigr)\longrightarrow1
\qquad\text{as }n\to\infty.
\]
\end{corollary}

\begin{proof}
Fix an arbitrary sequence satisfying the hypotheses, and abbreviate $G:=G_n$, $G_0:=G_n^0$, $d:=d_n$ and $p:=p_n$. Put
\[
        d':=(1-2\delta)d,
        \qquad
        c':=\frac{c-2\delta}{1-2\delta},
        \qquad\text{and}\qquad
        C':=\frac{C}{1-2\delta}.
\]
Since $0\le\delta<c/2$ and $c<1$, we have $d'>0$ and $c'\in(0,1)$.
The assumptions imply
\[
        e(G_0)\ge \frac{dn}{2}-\delta dn=\frac{d'n}{2},
\]
so $G_0$ has average degree at least $d'$. Every connected component of $G_0$ has vertex-cover number at most $Cd=C'd'$, and
$
        (1-c')d'=(1-c)d.
$
Let $K'$ and $d_0'$ be the constants supplied by Theorem~\ref{thm:smallcover} for the parameters $c',C',D$. Taking
\[
        K:=\frac{K'}{1-2\delta}
        \qquad\text{and}\qquad
        d_0:=\frac{d_0'}{1-2\delta},
\]
we have $d'\ge d_0'$ and $p\ge K'/d'$. Applying Theorem~\ref{thm:smallcover} to the sequence $(G_n^0,d_n',p_n)$, where $d_n'=(1-2\delta)d_n$, shows that $(G_0)_p$ is $\cF_{\le(1-c)d,D}$-universal with probability tending to one. Since $(G_0)_p\subseteq G_p$, the same is true of $G_p$.
\end{proof}

\medskip

\begin{remark}
Theorem~\ref{thm:transference} gives a density-sensitive resilient forest-universality statement, while Theorem~\ref{thm:dense}, Theorem~\ref{thm:smallcover} and Corollary~\ref{cor:nearsmallcover} verify the robust approximate Erd\H{o}s--S\'os conjecture for fixed maximum degree in the dense, small-cover and near-small-cover regimes. The remaining parts of Conjecture~\ref{conj:cmw} appear to require new ideas. In the proof of the robust Erd\H{o}s--Gallai theorem, the non-small-cover pieces are used to create a long depth-first-search jump, which is enough for a long cycle. For trees, a single long path-like object is not enough; one needs simultaneous branching capacity.
\end{remark}

The most immediate problem left open by this paper is whether the fixed-degree restriction can be removed from Theorem~\ref{thm:dense}.

\begin{question}\label{q:growingD}
For every $c\in(0,1)$ and $L\ge1$, is there $K=K(c,L)$ such that the following holds? If $F$ has at most $Ld$ vertices and average degree at least $d$, and if $T$ is a tree on at most $(1-c)d$ vertices with maximum degree at most $\Delta$, does $p\ge K\Delta/d$ imply that $F_p$ contains $T$ with probability tending to one?
\end{question}

A positive answer would give the full conjectured dependence on $\Delta$ in the small-cover regime.

%%%%%%%%%%%%%%%%%%%%%%%%%%%%%%%%%%%%%%%%%%%%%%%%%%%%%%%%%%%%%%%%%%%%

\end{document}